\documentclass{amsart}
\usepackage{amssymb, mathtools, enumerate, todonotes, cancel, tikz, float, wasysym}
\usepackage[hidelinks]{hyperref}
\usepackage{cleveref}
\usetikzlibrary{calc}

\DeclareMathOperator{\A}{A}
\DeclareMathOperator{\D}{D}
\DeclareMathOperator{\F}{F}
\DeclareMathOperator{\R}{R}
\DeclareMathOperator{\U}{U}

\DeclareMathOperator{\id}{1'}

\DeclareMathOperator{\down}{\downarrow}

\DeclareMathOperator{\maxit}{^\uparrow}
\DeclareMathOperator{\pref}{\sqcup}

\DeclarePairedDelimiter\abs{\lvert}{\rvert}

\newcommand{\compo}{\mathbin{;}}
\newcommand{\compl}[1]{\overline{#1}}
\newcommand{\bmeet}{\wedge}
\newcommand{\bjoin}{\vee}

\newcommand{\defn}[1]{\textbf{#1}}
\newcommand{\algebra}[1]{\mathfrak{#1}}

\theoremstyle{plain}
\newtheorem{theorem}{Theorem}[section]
\newtheorem{proposition}[theorem]{Proposition}
\newtheorem{lemma}[theorem]{Lemma}
\newtheorem{corollary}[theorem]{Corollary}

\theoremstyle{definition}
\newtheorem{definition}[theorem]{Definition}
\newtheorem{example}[theorem]{Example}

\title[The Finite Representation Property]{The Finite Representation Property for Composition, Intersection, Domain and Range}
\author{Brett McLean}
\address{Department of Computer Science, University College London, Gower Street, London WC1E 6BT}
\email{b.mclean@cs.ucl.ac.uk}
\thanks{The first author would like to thank his PhD supervisor Robin Hirsch for many helpful discussions.}
\author{Szabolcs Mikul\'as}
\address{Department of Computer Science and Information Systems, Birkbeck, University of London, Malet Street, London WC1E 7HX}
\email{szabolcs@dcs.bbk.ac.uk}
\begin{document}

\begin{abstract}
We prove that the finite representation property holds for representation by partial functions for the signature consisting of composition, intersection, domain and range and for any expansion of this signature by the antidomain, fixset, preferential union, maximum iterate and opposite operations. The proof shows that, for all these signatures, the size of base required is bounded by a double-exponential function of the size of the algebra. This establishes that representability of finite algebras is decidable for all these signatures. We also give an example of a signature for which the finite representation property fails to hold for representation by partial functions.
\end{abstract}

\maketitle

\section{Introduction}

The investigation of the abstract algebraic properties of partial functions involves studying the isomorphism class of algebras whose elements are partial functions and whose operations are some specified set of operations on partial functions---operations such as composition or intersection, for example. We refer to an algebra isomorphic to an algebra of partial functions as representable.

One of the primary aims is to determine how simply the class of representable algebras can be axiomatised and to find such an axiomatisation. Often, the representation classes have turned out to be axiomatisable by finitely many equations or quasi-equations \cite{schein, invitation, Jackson2003393, 1182.20058, DBLP:journals/ijac/JacksonS11, hirsch}.

Another question to ask is whether every finite representable algebra can be represented by partial functions on some finite set. Interest in this so-called finite representation property originates from its potential to help prove decidability of representability, which in turn can help give decidability of the equational or universal theories of the representation class.

Recently, Hirsch, Jackson and Mikul{\'a}s established the finite representation property for many signatures, but they leave the case for signatures containing the intersection, domain and range operations together open \cite{hirsch}.

In this paper we prove the finite representation property for the most significant group of outstanding signatures, which includes a signature containing all the most commonly considered operations on partial functions. From our proof we obtain a double-exponential bound on the size of base set required for a representation. It follows as a corollary that representability of finite algebras is decidable for all these signatures. As an additional observation, we give an example showing that there are signatures for which the finite representation property does not hold for representation by partial functions.

The results presented here originate with McLean \cite{McLean-arxiv}. The contribution of the second author is to translate the original proof of the finite representation property into a semantical setting, so that the presence of antidomain is not necessary.

\section{Algebras of Partial Functions}

In this section we give the fundamental definitions that are needed in order to state the results contained in this paper.

Given an algebra $\algebra{A}$, when we write $a \in \algebra{A}$ or say that $a$ is an element of $\algebra{A}$, we mean that $a$ is an element of the domain of $\algebra{A}$. We follow the convention that algebras are always nonempty.

\begin{definition}\label{first}
Let $\sigma$ be an algebraic signature whose symbols are a subset of $\{\compo, \bmeet, \D, \R, 0, \id, \A, \F, \pref, \maxit, \phantom{}^{-1} \}$. An \defn{algebra of partial functions} of the signature $\sigma$ is an algebra of the signature $\sigma$ whose elements are partial functions and with operations given by the set-theoretic operations on those partial functions described in the following.

Let $X$ be the union of the domains and ranges of all the partial functions occurring in an algebra $\algebra{A}$. We call $X$ the \defn{base} of $\algebra{A}$. 
The interpretations of the operations in $\sigma$ are given as follows:
\begin{itemize}
\item
the binary operation $\compo$ is \defn{composition} of partial functions:
\[f \compo g = \{(x,z) \in X^2 \mid \exists y \in X : (x, y) \in f\text{ and }(y, z) \in g\}\text{,}\]
that is, $(f\compo g)(x)=g(f(x))$,
\item
the binary operation $\wedge$ is \defn{intersection}:
\[f \bmeet g = \{(x,y) \in X^2 \mid (x, y) \in f\text{ and }(x, y) \in g\}\text{,}\]
\item
the unary operation $\D$ is the operation of taking the diagonal of the \defn{domain} of a function:
\[\D(f) = \{(x, x) \in X^2 \mid \exists y \in X : (x, y) \in f\}\text{,}\]
\item
the unary operation $\R$ is the operation of taking the diagonal of the \defn{range} of a function:
\[\R(f) = \{(y, y) \in X^2 \mid \exists x \in X : (x, y) \in f\}\text{,}\]
\item
the constant $0$ is the nowhere-defined \defn{empty function}:
\[0 = \emptyset\text{,}\]
\item
the constant $\id$ is the \defn{identity function} on $X$:
\[\id = \{(x, x) \in X^2 \}\text{,}\]
\item
the unary operation $\A$ is the operation of taking the diagonal of the \defn{antidomain} of a function\textemdash those points of $X$ where the function is not defined:
\[\A(f) = \{(x, x) \in X^2 \mid \cancel{\exists} y \in X : (x, y) \in f\}\text{,}\]
\item
the unary operation $\F$ is \defn{fixset}, the operation of taking the diagonal of the fixed points of a function:
\[\F(f) = \{(x, x) \in X^2 \mid (x, x) \in f\}\text{,}\]
\item
the binary operation $\pref$ is \defn{preferential union}:
\[(f \pref g)(x) =
\begin{cases}
f(x) & \text{if }f(x)\text{ defined}\\
g(x) & \text{if }f(x)\text{ undefined, but }g(x)\text{ defined}\\
\text{undefined} & \text{otherwise}
\end{cases}\]
\item
the unary operation $\maxit$ is the \defn{maximum iterate}:
\[ f^\uparrow = \bigcup_{0 \leq n < \omega} (f^n \compo \A(f))\text{,}\]
where $f^0 \coloneqq \id$ and $f^{n+1} \coloneqq f\compo f^n$,
\item
the unary operation $^{-1}$ is an operation we call \defn{opposite}:
\[ f^{-1} = \{(y, x) \in X^2 \mid (x, y) \in f\text{ and }((x', y) \in f \implies x = x')\}\text{.}\]
\end{itemize}
\end{definition}

The list of operations in \Cref{first} does not exhaust those that have been considered for partial functions, but does include the most commonly appearing operations.

\begin{definition}
Let $\algebra{A}$ be an algebra of one of the signatures permitted by \Cref{first}. A \defn{representation of $\algebra{A}$ by partial functions} is an isomorphism from $\algebra{A}$ to an algebra of partial functions of the same signature. If $\algebra{A}$ has a representation then we say it is \defn{representable}.
\end{definition}

In \cite{1182.20058}, Jackson and Stokes give a finite equational axiomatisation of the representation class for the signature $\{\compo, \bmeet, \D, \R\}$ and similarly for any expansion of this signature by operations in $\{0, \id, \F \}$.

In \cite{hirsch}, Hirsch, Jackson and Mikul{\'a}s give a finite equational axiomatisation of the representation class for the signature $\{\compo, \bmeet, \A, \R\}$ and similarly for any expansion of this signature by operations in $\{0, \id, \D, \F, \pref \}$. For expanded signatures containing the maximum iterate operation they give finite sets of axioms that, if we restrict attention to finite algebras, axiomatise the representable ones.

The operation that we call opposite is described in \cite{menger}, where Menger calls the concrete operation `bilateral inverse' and uses `opposite' to refer to an abstract operation intended to model this bilateral inverse. The opposite operation appears again in Schweizer and Sklar's \cite{ssi} and \cite{ssiii}, but thereafter does not appear to have received any further attention. In particular, for signatures containing opposite, axiomatisations of the representation classes remain to be found.

\begin{definition}
Let $\sigma$ be a signature. We say that $\sigma$ has the \defn{finite representation property} (for representation by partial functions) if whenever a finite algebra of the signature $\sigma$ is representable by partial functions, it is representable on a finite base.\footnote{This property has also been called the \emph{finite algebra on finite base property}.}
\end{definition}

In \cite{hirsch}, Hirsch, Jackson and Mikul{\'a}s establish the finite representation property for many signatures that are subsets of $\{\compo, \bmeet, \D, \R, 0, \id, \A, \F, \pref, \maxit, \phantom{}^{-1} \}$. Assuming composition is in the signature, they prove the finite representation property holds for any such signature not containing domain, any not containing range and almost all that do not contain intersection. This leaves one significant group of cases, which they highlight as an open problem: signatures containing $\{\compo, \bmeet, \D, \R \}$.

In this paper we prove that $\{\compo, \bmeet, \D, \R\}$ and any expansion of $\{\compo, \bmeet, \D, \R\}$ by operations that we have mentioned (including opposite) all have the finite representation property.

\section{Uniqueness of Presents and Futures}\label{section:unique}

In this section we derive some results about representations of finite algebras, which we will use in the following section to prove the finite representation property holds.

Throughout this section, let $\sigma$ be a signature with $\{\compo, \bmeet, \D, \R\} \subseteq \sigma \subseteq \{\compo, \bmeet, \D, \R, 0,\linebreak \id, \A, \F, \pref, \maxit, \phantom{}^{-1} \}$ and let $\algebra{A}$ be a finite representable $\sigma$-algebra. Since $\algebra{A}$ is representable, we may freely make use of basic properties of algebras of partial functions in the process of our deductions.

First note that the algebra $\algebra{A}$ is a meet-semilattice, with meet given by $\bmeet$. Whenever we treat $\algebra{A}$ as a poset, we are using the order induced by this semilattice. The set $\D(\algebra{A}) \coloneqq \{\D(a) \mid a \in \algebra{A}\}$ of \defn{domain elements} forms a subsemilattice of $\algebra{A}$.

In any representation $\theta$ of $\algebra{A}$ we have that $(x, y) \in \theta(a)$ if and only if $a$ is greater than or equal to the meet of the finite set $\{b \in \algebra{A} \mid (x, y) \in \theta(b)\}$. Hence we may identify each representation of $\algebra{A}$ with a particular edge-labelled directed graph (with reflexive edges). The label of an edge $(x, y)$ is the least element of $\{b \in \algebra{A} \mid (x, y) \in \theta(b)\}$. Since we take the base of a representation to be the union of the domains and ranges of all the partial functions, every vertex participates in some edge. Given that the domain and range {operations} are in the signature, this means that all vertices will have a reflexive edge.

The previous paragraph motivates our interest in the following type of object, of which representations of $\algebra{A}$ are a special case.

\begin{definition}
A \defn{network} over $\algebra{A}$ will be an edge-labelled directed graph, with labels drawn from $\algebra{A}$ and with a reflexive edge on every vertex.
\end{definition}

Given a network $N$, we will follow the usual convention of writing $x \in N$ to mean $x$ is a vertex of $N$ and we will denote the label of an edge $(x,y)$ by $N(x,y)$. We will speak of an element $b \in \algebra{A}$ holding on an edge $(x, y)$ when $N(x,y) \leq b$. We will call a vertex $x$ an \defn{$\boldsymbol{\alpha}$-vertex} if the reflexive edge at that vertex is labelled $\alpha$, that is if $N(x,x)=\alpha$.

Note that if $x$ is an $\alpha$-vertex of a representation then $\alpha$ is necessarily a domain element. Indeed if $\alpha$ holds on $(x, x)$ it follows that $\D(\alpha)$ holds on $(x, x)$. Since $\alpha$ is the least such element, we have $\alpha \leq \D(\alpha)$. It follows that $\alpha = \D(\alpha)$, by a property of partial functions.

\begin{definition}
Let $N$ be a network (over $\algebra{A}$) and let $W$ be a subset of the vertices of $N$. We define the \defn{future} of $W$ to be the subnetwork induced by the vertices reachable via an edge starting in $W$. We define the future of a vertex $x$ to be the future of the singleton set containing $x$.
\end{definition}

Since $\algebra{A}$ is finite and we are representing by partial functions, in a \emph{representation}, the future of a vertex must be a finite network.

Note also that in a representation, the taking of futures is a closure operator. Indeed, each $x \in W$ is reachable from $W$ via the reflexive edge at $x$. If there is an edge from $x \in W$ to $y$, labelled $a$, and from $y$ to $z$, labelled $b$, then $z$ is reachable via an edge starting in $W$, labelled $a \compo b$.

\begin{definition}
Let $N$ be a network (over $\algebra{A}$). The \defn{present} of a vertex $x$ of $N$ is the set of all vertices $y$ such that $y$ is in the future of $x$ and $x$ is in the future of $y$.
\end{definition}

We are interested in presents and futures because in \Cref{section:frp} we will describe how to use the presents and futures extant in representations in order to construct a representation on a finite base.

\begin{definition}
Let $a \in \algebra{A}$. If there exists a representation of $\algebra{A}$ in which $a$ labels an edge, then we will call $a$ \defn{realisable}.
\end{definition}

\begin{proposition}\label{unique_futures}
For any realisable domain element $\alpha \in \algebra{A}$, any two $\alpha$-vertices from any two representations have isomorphic futures.
\end{proposition}

\begin{proof}
Let $x$ be an $\alpha$-vertex: $\alpha=N(x,x)$. We claim that there is an edge starting at $x$ labelled with $a$ if and only if $\D(a) = \alpha$. 

Suppose first that $\D(a) = \alpha$. Then, by the definition of the domain operation, there must be an edge starting at $x$ labelled with some $b \leq a$. Then $\D(b)$ must hold on $(x, x)$ and so $\D(a) \leq \D(b)$. From $b \leq a$ and $\D(a) \leq \D(b)$ it follows that $a = b$, by a property of partial functions. Hence there is an $a$-labelled edge starting at $x$.

Conversely, suppose there is an edge labelled with $a$ starting at $x$, ending at $y$ say. Then $\alpha \compo a$ holds on $(x, y)$ and so $a \leq \alpha \compo a$. Since $\alpha$ is a domain element, this implies $\D(a) \leq \alpha$, by a property of partial functions. But $\D(a)$ holds on $(x, x)$ and so $\alpha \leq \D(a)$, since $\alpha=N(x, x)$. We conclude that $\D(a) = \alpha$.

Note that as the elements of $\algebra{A}$ are represented by partial functions, there cannot be multiple edges starting at $x$ on which the same element holds. In particular there cannot be multiple edges with the same label. We therefore now know that for any $\alpha$-vertex $x$, the edges starting at $x$ are precisely a single edge labelled $a$ for every $a$ with $\D(a) = \alpha$ (the $\alpha$-labelled edge being the reflexive one). So we have an obvious candidate for an isomorphism between the futures of $\alpha$-vertices: given two $\alpha$-vertices $x$ and $x'$ in networks $N$ and $N'$ respectively, we let $y\mapsto y'$ if and only if $N(x,y)=N'(x',y')$.

For $a$ and $b$ with $\D(a) = \D(b) = \alpha$, let $(x, y)$ and $(x, z)$ be the two edges starting at $x$ and labelled by $a$ and $b$ respectively. To show that the future of $x$ is uniquely determined up to isomorphism, we only need show that the set of elements of $\algebra{A}$ holding on $(y, z)$ is uniquely determined. We claim that 
\begin{equation}\label{eq-comp}
\text{if $N(x,y)=a$ and $N(x,z)=b$, then $c$ holds on $(y, z)$ if and only if $a \compo c = b$}, 
\end{equation}
which gives us what we want.

Suppose first that $a \compo c = b$. Then as $(x, y)$ is the \emph{unique} edge starting at $x$ on which $a$ holds, $c$ must hold on $(y, z)$ in order that composition be represented correctly. Conversely, suppose that $c$ holds on $(y, z)$. Then $a \compo c$ holds on $(x, z)$ and so $b \leq a \compo c$. But $\D(a \compo c) \leq \D(a)$ is valid in all representable algebras and $\D(a) = \D(b)$. From $b \leq a \compo c$ and $\D(a \compo c) \leq \D(b)$ we may conclude $a \compo c = b$, by a property of partial functions.
\end{proof}

For realisable domain elements $\alpha$ and $\beta$, write $\alpha \lesssim \beta$ if in a (or every) representation of $\algebra{A}$ there is a $\beta$-vertex in the future of every $\alpha$-vertex, or equivalently if there exists an $a \in \algebra{A}$ with $\D(a) = \alpha$ and $\R(a) = \beta$. Then $\lesssim$ is easily seen to be a preorder on the realisable domain elements.

\begin{proposition}\label{lemma:complete}
For any realisable domain elements $\alpha, \beta \in \algebra{A}$, if $\alpha$ and $\beta$ are $\lesssim$-equivalent, then any $\alpha$-vertex and any $\beta$-vertex from any two representations have isomorphic futures.
\end{proposition}

\begin{proof}
Let $x$ be an $\alpha$-vertex, from some representation of $\algebra{A}$. As $\alpha \lesssim \beta$, in the same representation there is a $\beta$-vertex, $y$ say, in the future of $x$. As $\beta \lesssim \alpha$ there is an $\alpha$-vertex, $z$ say, in the future of $y$. Hence $z$ is in the future of $x$, meaning that the future of $z$ is a subnetwork of the future of $x$. But these are finite isomorphic objects and therefore equal. So $x$ is in the future of $z$ and therefore $x$ is in the future of $y$. Hence the futures of the $\alpha$-vertex $x$ and the $\beta$-vertex $y$ are equal in this representation. By \Cref{unique_futures}, we conclude the required result.
\end{proof}

In a representation, the present of an $\alpha$-vertex $x$ is always the initial, strongly connected component of $x$'s future---the one that can `see' the entire future of $x$. So we get the following immediate corollary of \Cref{lemma:complete}.

\begin{corollary}\label{corollary:present}
For any realisable domain elements $\alpha, \beta \in \algebra{A}$, if $\alpha$ and $\beta$ are $\lesssim$-equivalent, then any $\alpha$-vertex and any $\beta$-vertex from any two representations have isomorphic presents.
\end{corollary}

Given a $\lesssim$-equivalence class $E$, we will speak of `the future of $E$' to mean the unique isomorphism class of futures of $\alpha$-vertices in representations, for any $\alpha \in E$. Similarly for `the present of $E$'.

\section{The Finite Representation Property}\label{section:frp}

In this section, we prove our main result: the finite representation property holds for all the signatures we are interested in. We then use our proof to calculate an upper bound on the size of base required.

To construct a representation over a finite base we will use the realisable domain elements and the preorder $\lesssim$ on them defined in \Cref{section:unique}. Recall that the realisable domain elements are the domain elements appearing as reflexive-edge labels in some representation. 

We start with a lemma that is little more than a translation of the definition of a representation into the language of graphs, but which gives us an opportunity to state exactly what is needed in order for a network to be a representation.

\begin{lemma}\label{lemma:network}
Let $\algebra{A}$ be a finite representable algebra of a signature $\sigma$ with $\{\compo, \bmeet, \D, \R\} \linebreak\subseteq \sigma \subseteq \{\compo, \bmeet, \D, \R, 0, \id, \A, \F, \pref, \maxit, \phantom{}^{-1}\}$ and let $N$ be a network (over $\algebra{A}$). Then $N$ is a representation of $\algebra{A}$ if and only if the following conditions are satisfied.
\begin{enumerate}[(i)]
\item\label{functional}
\textbf{(Relations are functions)} For any vertex $x$ of $N$ and any $a \in \algebra{A}$ there is at most one edge starting at $x$ on which $a$ holds.
\item\label{operations}
\textbf{(Operations represented correctly)} Let $*$ be an operation in the signature (excluding $\bmeet$). Then (assuming for simplicity of presentation that $*$ is a binary operation) if we apply the appropriate set-theoretic operation to the set of edges where $a$ holds and the set of edges where $b$ holds then we get precisely the set of edges where $a * b$ holds:
\[
\{(x,y)\mid N(x,y)\le a\}*\{(x',y')\mid N(x',y')\le b\}=\{(x'',y'')\mid N(x'',y'')\le a * b\}.
\]
\item\label{faithful}
\textbf{(Faithful)} For every $a, b \in \algebra{A}$ with $a \nleq b$, there is an edge of $N$ on which $a$ holds, but $b$ does not.
\end{enumerate}
\end{lemma}

\begin{proof}
Routine.
\end{proof}

We also require the following definition.

\begin{definition}
Let $\algebra{A}$ be a finite representable algebra of a signature $\sigma$ with $\{\compo, \bmeet, \D, \R\} \subseteq \sigma \subseteq \{\compo, \bmeet, \D, \R, 0, \id, \A, \F, \pref, \maxit, \phantom{}^{-1}\}$. From the relation $\lesssim$ defined in \Cref{section:unique}, form the partial order of $\lesssim$-equivalence classes (of realisable domain elements of $\algebra{A}$). The \defn{depth} of a $\lesssim$-equivalence class $E$ will be the length of the longest increasing chain in this partial order, starting at $E$. (We take the length of a chain to be one fewer than the number of elements it contains, so a maximal $\lesssim$-equivalence class has depth zero.) Since $\algebra{A}$ is finite, the depth of every $\lesssim$-equivalence class is finite and bounded by the size of $\algebra{A}$. The depth of a realisable domain element will be the depth of its $\lesssim$-equivalence class.
\end{definition}

We are now ready to prove our main result, but note that the following theorem does not cover signatures containing opposite.

\begin{theorem}\label{main}
The finite representation property holds for representation by partial functions for any signature $\sigma$ with $\{\compo, \bmeet, \D, \R\} \subseteq \sigma \subseteq \{\compo, \bmeet, \D, \R, 0, \id, \A, \F, \pref, \maxit \}$.
\end{theorem}

\begin{proof}
Let $\algebra{A}$ be a finite representable algebra of one of the signatures under consideration. We construct a finite network $N$ step by step, by adding copies of the present of $\lesssim$-equivalence classes of increasing depths. Then we argue that the resulting network is a representation of $\algebra{A}$. The idea of the proof is to always `add everything we can, $\abs{\algebra{A}}$ times'.

Assume inductively that we have carried out steps $0, \ldots, n-1$ of our construction, giving us the network $N_{n-1}$. We form $N_n$ as follows. (For the base case of this induction, we let $N_{-1}$ be the empty network.) Let $E$ be a $\lesssim$-equivalence class of depth $n$, and $P$ a copy of the present of $E$. A choice of edges from $P$ to $N_{n-1}$ labelled by elements of $\algebra{A}$ is \defn{allowable} if adding $P$ and these edges to $N_{n-1}$ would make the future of $P$ in the extended network isomorphic to the future of $E$. For every allowable choice, to $N_{n-1}$ we add $\abs{\algebra{A}}$ copies of both $P$ and the edges from $P$ specified by the choice. The network $N_n$ is the network we have once we have done this for every $\lesssim$-equivalence class of depth $n$.

Note that the order that $\lesssim$-equivalence classes of a given depth $n$ are processed is immaterial, since no allowable choice could have an edge ending at a vertex that had not been in $N_{n-1}$. By induction, each $N_n$ is finite: assume that $N_{n-1}$ is finite; then as each $P$ is also finite and $\algebra{A}$ is finite we see that the number of allowable choices is finite, so $N_n$ is finite. We take $N$ to be $N_M$, where $M$ is the maximum depth of any $\lesssim$-equivalence class of $\algebra{A}$.

For any vertex $x$ of $N$, the future of $x$ during the various stages of the construction of $N$ is unaltered once $x$ has been added to the construction. So the future of $x$ in $N$ is isomorphic to the future of some vertex in a representation of $\algebra{A}$, since this is true at the moment that $x$ is added to the construction, by the definition of an allowable choice.

The next lemma will ensure that allowable choices always exist.
Let $N'$ be the underlying network for a representation of $\algebra{A}$.
Fix a vertex $x\in N'$ and write $F$ for the future of $x$ in $N'$.
Define $F_n$ to be the subnetwork of $F$ induced by vertices $y$ such that the depth of $N'(y,y)$ is at most $n$.

\begin{lemma}\label{lem-ex}
For every $n \geq -1$, there is an embedding $f_n\colon F_n \hookrightarrow N_n$.
Moreover, if $G$ is any future-closed subset of $F$ and $g \colon G \hookrightarrow N$ is an embedding, then $f_n$ can be chosen so that it agrees with $g$ wherever $f_n$ and $g$ are both defined.
\end{lemma}

\begin{proof}
We use induction on $n$.
As before, the base case for the induction is depth $-1$, so we define $f_{-1}$ to be the empty map.

For $n > -1$, suppose we have an embedding $f_{n-1}\colon F_{n-1}\hookrightarrow N_{n-1}$, a future-closed subset $G$ of $F$ and an embedding $g \colon G \hookrightarrow N$ such that $f_{n-1}$ and $g$ agree where both are defined. We can form $f_n$, an extension of $f_{n-1}$, as follows. 
First use $g$ to define an intermediate extension $f'_{n-1}$ of $f_{n-1}$ to those vertices in $(F_n \setminus F_{n-1}) \cap G$, that is:
\[
f'_{n-1}(y) =
\begin{cases}
f_{n-1}(y) & \text{if }y\in F_{n-1}\\
g(y) & \text{if }y\in (F_n \setminus F_{n-1}) \cap G.
\end{cases}
\]
Using the assumption that $G$ is future closed, we will show that this intermediate extension $f'_{n-1}$ is still an embedding. Observe that for any $y \in (F_n \setminus F_{n-1}) \cap G$, the future of $g(y)$ in $N$ is isomorphic to the future of $y$ in $N'$, since the future of any vertex of $N$ is isomorphic to the future, in any representation, of any vertex with the same reflexive-edge label. Hence $f'_{n-1}$ maps elements of $(F_n \setminus F_{n-1}) \cap G$ to elements of $N_n \setminus N_{n-1}$, from which we see that $f'_{n-1}$ is injective. Now for $f'_{n-1}$ to be an embedding of $F_{n-1}\cup (F_n\cap G)$, we need to show that for arbitrary $y, z \in F_{n-1}\cup (F_n\cap G)$ and $a \in \algebra{A}$,  there is an edge labelled $a$ from $y$ to $z$ if and only if there is an edge labelled $a$ from $f'_{n-1}(y)$ to $f'_{n-1}(z)$. If $y, z \in F_{n-1}$ then we use that $f_{n-1}$ is an embedding. Similarly, if $y, z \in (F_n \setminus F_{n-1}) \cap G$ then we use that $g$ is an embedding. If $y \in F_{n-1}$ and $z \in (F_n \setminus F_{n-1}) \cap G$ then there is no edge $(y, z)$, because there are no edges from $F_{n-1}$ to $F_n \setminus F_{n-1}$, nor is there an edge $(f'_{n-1}(y), f'_{n-1}(z))$, because there are no edges from $N_{n-1}$ to $N_n \setminus N_{n-1}$.
It remains to consider the case when $y \in (F_n \setminus F_{n-1}) \cap G$ and $z\in F_{n-1}$.

First assume there is an edge from $y$ to $z$ labelled $a$. Then $z \in G$, since $G$ is future closed. Hence $f'_{n-1}(z) = f_{n-1}(z) = g(z)$, by the inductive hypothesis, and the edge $(g(y), g(z))$ is labelled $a$ in $N$, since $g$ is an embedding. 

Conversely, assume there is an edge from $g(y)$ to $f'_{n-1}(z)$ labelled $a$. Then $f'_{n-1}(z)$ is in the future of $g(y)$ in $N$. Since the future of $g(y)$ in $N$ is isomorphic to the future of $y$ in the representation $N'$, we see firstly that $f'_{n-1}(z)$ is the unique vertex of $N$ being the end of an $a$-labelled edge starting at $g(y)$. Secondly, in $N'$ there is a unique $a$-labelled edge, $(y, z')$ say, starting at $y$. Then $z' \in G\cap F_n$, since $G$ and $F_n$ are future closed, and  $g(z') = f'_{n-1}(z)$, since $g$ is an embedding. By injectivity of $f'_{n-1}$, we have $z' = z$ and so $N'(y,z)=a$, as desired. This completes the proof that $f'_{n-1}$ is an embedding.

The remaining vertices we need to extend to are partitioned into copies of the present of various $\lesssim$-equivalence classes of depth $n$. 
Fix one copy $P$ in $F_n$ of one of these equivalence classes. Then $P$ and $f_{n-1}$ (being an embedding of $F_{n-1}$ into $N_{n-1}$) together specify an allowable choice of edges from $P$ to $N_{n-1}$.
Since every allowable choice has been replicated $\abs{\algebra{A}}$ times during each step of the construction of $N_{n}$, this provides not just one but $\abs{\algebra{A}}$ possible ways to extend $f_{n-1}$ to $P$ and to the edges starting in $P$.
The number of vertices starting at $x$ in $F$ is bounded by the number of elements of $\algebra{A}$. So $F$, and therefore $F_n$, certainly contain no more than $\abs{\algebra{A}}$ copies of the present of any $\lesssim$-equivalence class of depth $n$ (including any copies in $G$). Hence there exists a way to extend $f'_{n-1}$ to all these copies simultaneously.
\end{proof}

It remains to show that $N$ is a representation of $\algebra{A}$, so we need to show that $N$ satisfies the conditions of \Cref{lemma:network}. It is easy to see that the relations are functions. From the fact that the future of any vertex of $N$ is isomorphic to the future of some vertex in a representation of $\algebra{A}$, it follows that there is at most one edge starting at $x$ on which any given $a \in \algebra{A}$ holds.

Next we need to show that the operations are represented correctly by $N$. With the exception of range, all the operations are straightforward and similar to show. We again rely on the fact that for any vertex $x$ in $N$, the future of $x$ is isomorphic to the future of some vertex in a representation of $\algebra{A}$.

To see that composition is represented correctly, suppose first that $a$ holds on $(x, y)$ and $b$ holds on $(y, z)$. Then as the future of $x$ matches the future of some vertex in a representation, $a \compo b$ holds on $(x, z)$. Conversely, suppose that $a \compo b$ holds on $(x, z)$. Then again by matching $x$ with a vertex in a representation, we know there is a $y$ such that $a$ holds on $(x, y)$ and $b$ holds on $(y, z)$.

To see that domain is represented correctly, suppose first that $\D(a)$ holds on $(x, y)$. Then by matching $x$ with a vertex in a representation, we know both that $x = y$ and that there is an edge starting at $x$ on which $a$ holds. Conversely, suppose that $a$ holds on an edge $(x, y)$. Then by matching $x$, we see that $\D(a)$ must hold on $(x, x)$.

If $0$ is in the signature, then no edge in any representation of $\algebra{A}$ is labelled with $0$. Hence $0$ does not hold on any edge in $N$ and so $N$ represents $0$ correctly. If $\id$ is in the signature then in any representation, $\id$ holds on all reflexive edges and no others. Hence the same is true of $N$ and so $N$ represents $\id$ correctly.

To see that antidomain is represented correctly if it is in the signature, suppose first that $\A(a)$ holds on $(x, y)$. Then by matching $x$ with a vertex in a representation, we know both that $x = y$ and that there is no edge starting at $x$ on which $a$ holds. Conversely, suppose there is no edge starting at $x$ on which $a$ holds. Then by matching $x$, we see that $\A(a)$ must hold on $(x, x)$.

To see that fixset is represented correctly if it is in the signature, suppose first that $\F(a)$ holds on $(x, y)$. Then by matching $x$ with a vertex in a representation, we know both that $x = y$ and that $a$ holds on $(x, x)$. Conversely, suppose $a$ holds on $(x, x)$. Then by matching $x$, we see that $\F(a)$ must hold on $(x, x)$.

To see that preferential union is represented correctly if it is in the signature, suppose first that $a \pref b$ holds on $(x, y)$. Then by matching $x$ with a vertex in a representation, we know that on $(x, y)$ either $a$ holds, or $a$ does not hold and $b$ does. Conversely, suppose that on $(x, y)$ either $a$ holds, or $a$ does not hold and $b$ does. Then by matching $x$, we see that $a \pref b$ must hold on $(x, y)$.

To see that maximum iterate is represented correctly if it is in the signature, suppose first that $a^\uparrow$ holds on $(x_0, x_n)$. Then by matching $x_0$ with a vertex in a representation, we know there exist $x_0, x_1, \ldots, x_n$ such that $a$ holds on each $x_i, x_{i+1}$ and there is no edge starting at $x_n$ on which $a$ holds. Conversely, suppose there exist $x_0, x_1, \ldots, x_n$ such that $a$ holds on each $(x_i, x_{i+1})$ and there is no edge starting at $x_n$ on which $a$ holds. Then by matching $x_0$, we see that $a^\uparrow$ must hold on $(x_0, x_{i+1})$.

One direction of range being represented correctly is clear: if $N$ has an edge from $x$ to $y$ on which $a$ holds, then $\R(a)$ will hold on the reflexive edge at $y$. For the other direction, suppose that $\R(a)$ holds on a vertex $y$ in $N$ and let $\beta$ be the label of $y$. Then we know that we can find a $\beta$-vertex, $y'$ say, in a representation and that $\R(a)$ will hold on $y'$. So there is an edge $(x', y')$ in this representation on which $a$ holds. Since the future of $y'$ is isomorphic to the future of $y$ via an isomorphism sending $y'$ to $y$, there is an embedding of the future of $y'$ into $N$ sending $y'$ to $y$. Then \Cref{lem-ex} ensures we can embed the future of $x'$ into $N$ in such a way that $y'$ is mapped to $y$ and so there is an $a$-labelled edge ending at $y$.

For the condition that $N$ be faithful, consider any $a, b \in \algebra{A}$ with $a \nleq b$. Then as $\algebra{A}$ is representable, there certainly exists some realisable $\alpha \in \algebra{A}$ having an edge in its future on which $a$ holds, but $b$ does not. Since the futures of $\alpha$-vertices in $N$ are isomorphic to the futures of $\alpha$-vertices in representations, it suffices to show that for every $\lesssim$-equivalence class $E$, a nonzero number of copies of the present of $E$ are added at the appropriate stage of the construction. But this is obvious, by \Cref{lem-ex}.
\end{proof}

With a little more work we can expand the list of operations to include opposite.

\begin{theorem}\label{opposite}
The finite representation property holds for representation by partial functions for any signature $\sigma$ with $\{\compo, \bmeet, \D, \R\} \subseteq \sigma \subseteq \{\compo, \bmeet, \D, \R, 0, \id, \A, \F, \pref, \maxit, \phantom{}^{-1}\}$.
\end{theorem}

\begin{proof}
Let $\algebra{A}$ be a finite representable algebra of one of the specified signatures. We may assume that $\algebra{A}$ has more than one element, as the one-element algebra is representable using the empty set as a base. We will argue that if the signature contains opposite, then the network $N$ described in the proof of \Cref{main} represents opposite correctly.

Suppose first that $a^{-1}$ holds on $(y, x)$. We want to show that $(x, y)$ is the unique edge ending at $y$ on which $a$ holds. As the future of $y$ matches the future of some vertex in a representation, we know that $a$ holds on $(x, y)$. To show that $(x, y)$ is the unique such edge, suppose $a$ holds on $(x', y)$. Then $x$, $x'$ and $y$ are all in the future of $x'$. So in the future of $x'$ we have: $a^{-1}$ holding on $(y, x)$ and $a$ holding on $(x, y)$ and $(x', y)$. As the future of $x'$ matches the future of some vertex in a representation, it follows that $x = x'$, as required.

Conversely, suppose that $(x, y)$ is the unique edge ending at $y$ on which $a$ holds. We want to show that $a^{-1}$ holds on $(y, x)$. Let $\alpha$ be the label of the reflexive edge at $x$, let $\beta$ be the label of the reflexive edge at $y$ and let $b$ be the label of $(x, y)$. First note that if $\alpha$ were in a deeper $\lesssim$-equivalence class than $\beta$, then, because of the way $N$ is constructed, there would be at least $|\algebra{A}|$ edges ending at $y$ on which $a$ holds. Hence $\alpha$ and $\beta$ are in the same $\lesssim$-equivalence class.

Now the present of $x$ is isomorphic to the present of any $\alpha$-vertex in any representation of $\algebra{A}$. So it suffices to show that for an $\alpha$-vertex $x'$ in a representation of $\algebra{A}$, if $(x', y')$ is the $b$-labelled edge from $x'$ to a $\beta$-vertex, then $a^{-1}$ holds on $(y', x')$. Being situated in a representation, we can show this by proving that $(x', y')$ is the unique edge ending at $y'$ on which $a$ holds.

Suppose then that $a$ holds on $(z', y')$ and let $\gamma$ be the label of the reflexive edge at $z'$. Suppose $\gamma \neq \alpha$. We saw, in proving that range is represented correctly, that we can embed the future of $z'$ into $N$ in such a way that $y'$ is mapped to $y$. So there is an edge starting at a $\gamma$-vertex and ending at $y$ on which $a$ holds. But this is a contradiction, as the edge $(x, y)$, starting at an $\alpha$-vertex, is supposed to be the unique edge ending at $y$ on which $a$ holds. We conclude that $\gamma = \alpha$ and hence $z'$ is in the present of $y'$, since $\alpha$ and $\beta$ are in the same $\lesssim$-equivalence class. We must now have $x' = z'$, for otherwise the present of $x'$ would feature two distinct edges ending at $y'$ on which $a$ holds. We know this not to be the case, by comparison with $y$, in the present of $x$. Hence $(x', y')$ is the unique edge ending at $y'$ on which $a$ holds, as required.
\end{proof}

Given some representation $\theta$ of an algebra $\algebra{A}$, we could give an alternate definition of the realisable elements of $\algebra{A}$ as those appearing as edge labels in \emph{the particular representation $\theta$}, rather than just in any representation. Then our proofs of \Cref{main} and \Cref{opposite} would work equally well. However, with the definition we gave, the constructed representation is in a sense the richest possible, in that if it is possible for an element to appear as a label in a representation, then it appears as a label in the constructed representation.

It is clear that from the proof of \Cref{main} we can extract a bound on the size required for the base.

\begin{proposition}\label{bound}
For any signature $\sigma$ with $\{\compo, \bmeet, \D, \R\} \subseteq \sigma \subseteq \{\compo, \bmeet, \D, \R, 0, \id, \A, \F, \linebreak\pref, \maxit, \phantom{}^{-1} \}$ every finite $\sigma$-algebra $\algebra{A}$ is representable over a base of size
\Large\[\abs{\algebra{A}}^{\abs{\algebra{A}}^{O(\abs{\algebra{A}}^\frac{1}{2})}}\text{.}\]
\end{proposition}

\begin{proof}
We may assume $\abs{\algebra{A}} \geq 2$. Let $N$ and $(N_n)_{n\geq-1}$ be as in the proofs of \Cref{main} and \Cref{opposite}. Let $E$ be a $\lesssim$-equivalence class of depth $n$ and let $P$ be a copy of the present of $E$. An allowable choice from $P$ to $N_{n-1}$ is determined by the labelled edges from a single vertex of $P$, since it follows from claim \eqref{eq-comp} in the proof of \Cref{unique_futures} that if $a$ is the label of an edge $(x, y)$ and $b$ is the label of an edge $(y, z)$ then $a \compo b$ is the label of the edge $(x, z)$. There are at most $\abs{\algebra{A}}$ labels, so at most $\abs{N_{n-1}}^{\abs{\algebra{A}}}$ allowable choices (unless $E$ is of depth $0$, in which case there is a single allowable choice). When $N_n$ is constructed from $N_{n-1}$, for each allowable choice, $\abs{\algebra{A}}$ copies of $P$ are added, so $\abs{\algebra{A}}\abs{P}$ vertices are added. The sum, over all $\lesssim$-equivalence classes of depth $n$, of the number of vertices in the present of each class, is at most $\abs{\algebra{A}}$. Hence at most $\abs{\algebra{A}}^2\abs{N_{n-1}}^{\abs{\algebra{A}}}$ vertices are added when $N_n$ is constructed from $N_{n-1}$. We obtain
\[
\begin{split}
\abs{N_{0}} &\leq \abs{\algebra{A}}^2\text{,} \\
\abs{N_n} &\leq \abs{N_{n-1}} + \abs{\algebra{A}}^2\abs{N_{n-1}}^{\abs{\algebra{A}}}\text{ for }n\geq1\text{,}
\end{split}
\]
from which it is provable by induction that
\[\abs{N_n} \leq \abs{\algebra{A}}^{2(n+1)\abs{\algebra{A}}^n}\text{.}\]
At least $n(n+1)/2$ distinct elements of $\algebra{A}$ are required in order for there to be a $\lesssim$-equivalence class of depth $n$, since composition is in the signature. So the construction of $N$ is completed by a depth that is $O(\abs{\algebra{A}}^\frac{1}{2})$. Hence
\Large\[\abs{N} = \abs{\algebra{A}}^{2(O(\abs{\algebra{A}}^\frac{1}{2})+1)\abs{\algebra{A}}^{O(\abs{\algebra{A}}^\frac{1}{2})}} = \abs{\algebra{A}}^{\abs{\algebra{A}}^{O(\abs{\algebra{A}}^\frac{1}{2})}}\text{.}\qedhere\]
\end{proof}

For comparison, note that in \cite{hirsch}, whenever a signature is shown to have the finite representation property, a bound on the size required for the base is derived that has either polynomial or exponential asymptotic growth.

We mentioned in the introduction that proving the finite representation property can help show that representability of finite algebras is decidable. The most direct way this can happen is by finding a (computable) bound on the size required for a representation. Then the representability of a finite algebra can be decided by searching for an isomorph amongst the concrete algebras with bases no larger than the bound. 

For most of the signatures that we have considered, decidability has already been established, because finite equational or quasiequational axiomatisations of the representation classes (or at least the finite representable algebras) are known. However this is not the case for some of our signatures. Specifically, the antidomain-free expansions of $\{\compo, \bmeet, \D, \R\}$ by $\maxit$ and/or $\pref$ and also any of the signatures containing opposite. So it is worth stating the following corollary of \Cref{bound}.

\begin{corollary}
Representability of finite algebras by partial functions is decidable for any signature $\sigma$ with $\{\compo, \bmeet, \D, \R\} \subseteq \sigma \subseteq \{\compo, \bmeet, \D, \R, 0, \id, \A, \F, \pref, \maxit, \phantom{}^{-1} \}$.
\end{corollary}

\section{Entirely Algebraic Constructions}

Most of the construction detailed in \Cref{section:frp} can be carried out based only on direct inspection of the algebra under consideration. However we noted that the construction does depend in one respect on information contained in representations of the algebra: the representations determine which are the realisable domain elements. We also noted that our construction works equally well if our realisable elements are those appearing as edge labels in one particular representation. So if we were to give an algebraic characterisation of the elements appearing as reflexive-edge labels in a particular representation, we would have a method of constructing a representation on a finite base using only algebraic properties of the algebra. Giving such characterisations, for certain signatures, is precisely what we do in this section.

We first mention the signature $\{\compo, \bmeet, \D, \R\}$ and expansions of this signature by operations in $\{0, \id, \F\}$. The representation that Jackson and Stokes give in \cite{1182.20058} for these signatures uses for the base of the representation Schein's `permissible sequences', as originally described in \cite{schein}. A permissible sequence, is a sequence $(a_1, b_1, \ldots, a_n, b_n, a_{n+1})$ with $\R(a_i) = \R(b_i)$ and $\D(b_i) = \D(b_{i+1})$ for each $i$ (and $0$ cannot participate in a sequence if it is in the signature). There is an edge on which $c$ holds, starting at such a sequence, if and only if $\D(c) \geq \R(a_{n+1})$. Hence for Schein's representation we can identify the elements labelling reflexive edges quite easily: they are those of the form $\R(a)$, for some $a$ (excluding $0$ if it is in the signature).

Now we examine the signature $\{\compo, \bmeet, \A, \R\}$. An arbitrary representable $\{\compo, \bmeet, \A, \linebreak\R\}$-algebra, $\algebra{A}$, has a least element, $0$, given by $\A(a) \compo a$ for any $a \in \algebra{A}$ and any representation of $\algebra{A}$ must represent $0$ by the empty set. We can define $\D \coloneqq \A^2$ and in any representation this must be represented by the domain operation.

The down-set of any element $a \in \algebra{A}$ forms a Boolean algebra using the meet operation of $\algebra{A}$ and with complementation given by $\compl{b} \coloneqq \A(b) \compo a$. Any representation $\theta$ of $\algebra{A}$ by partial functions restricts to a representation of each $\down a$ as a field of sets over $\theta(a)$. From this we see that $\theta$ turns any finite joins in $\algebra{A}$ into unions.

\begin{definition}
Let $\algebra{P}$ be a poset with a least element, $0$. An \defn{atom} of $\algebra{P}$ is a minimal nonzero element of $\algebra{P}$. We say that $\algebra{P}$ is \defn{atomic} if every nonzero element is greater than or equal to an atom.
\end{definition}

A \emph{finite} representable $\{\compo, \bmeet, \A, \R\}$-algebra, $\algebra{A}$, is necessarily atomic. Any $a \in \algebra{A}$ can be expressed as a finite join of atoms of $\algebra{A}$ since, given $0 < a < b$, we can split $b$ as $b = a \bjoin (\A(a) \compo b)$.

From the preceeding discussion, we see that in any representation of a $\{\compo, \bmeet, \A, \R\}$-algebra, for any edge a finite sum of atoms holds, so at least one of the atoms holds, as finite joins are represented by unions. We know that at most one atom holds, since the meet of two distinct atoms is $0$, which can never hold on an edge. Hence a unique atom holds on each edge and necessarily labels the edge. In every representation every atom must appear as a label, otherwise it is not separated from $0$. We conclude that in any representation the elements appearing as edge labels are precisely the atoms and so the elements labelling reflexive edges are precisely the atomic domain elements. Hence for the signature $\{\compo, \bmeet, \A, \R\}$ the realisable domain elements are the atomic domain elements. This also applies to any expansion of this signature by operations we have mentioned.

The purpose of the next example is simply to illustrate that, unlike Boolean algebras for example, the set of atoms in a finite representable $\{\compo, \bmeet, \A, \R\}$-algebra can be almost as large as the algebra itself. Hence applying the knowledge that the number of labels is at most the number of atoms to the calculation in \Cref{bound}, does not improve the bound.

\begin{example}
Let $G$ be any finite group. We can make $G \cup \{0\}$ into an algebra of the signature $\{\compo, \bmeet, \A, \R\}$ by using the group operation for composition (and $g \compo 0 = 0 \compo g = 0$ for all $g$) and defining $g \bmeet h = 0$ unless $g = h$, every antidomain of a nonzero element to be $0$ (and $\A(0) = e$, the group identity) and every range of a nonzero element to be $e$ (and $\R(0) = 0$). Then every nonzero element of $G \cup \{0\}$ is an atom. Augmenting the Cayley representation of $G$ (the representation $\theta(g)(h)= hg$) by setting $\theta(0) = \emptyset$ demonstrates that $G \cup \{0\}$ is representable.
\end{example}

\section{Failure of the Finite Representation Property}

Finally, one might reasonably wonder if it is possible for the finite representation property \emph{not} to hold for algebras of partial functions. After all, for every signature for which it has been settled, the finite representation property has been shown to hold. We finish with a simple example showing that we can indeed force a finite representable algebra of partial factions to fail to have representations over finite bases.

\begin{example}\label{example}
Let $\U$ be the unary operation on partial functions given by
\[\U(f) = \{(y, y) \in X^2 \mid \exists! x \in X : (x, y) \in f\}\text{.}\]
Let $\algebra{F}$ be the algebra of partial functions, of the signature $\{\compo, \bmeet, \D, \R, \U\}$ and with base $\omega \times 2$, containing the following five elements.
\begin{itemize}
\item
$0$, the empty function,
\item
$d$, the identity function on $\omega \times \{0\}$,
\item
$r$, the identity function on $\omega \times \{1\}$,
\item
$f$, the function with domain $d$ and range $r$ sending each $(n, 0)$ to $(n, 1)$,
\item
$g$, a function with domain $d$ and range $r$ such that each $(n, 1) \in \omega \times \{1\}$ has precisely two $g$-preimages: the least two elements of $\omega \times \{0\}$ that are neither the $f$-preimage $(n, 0)$ nor $g$-preimages of $(m, 1)$ for $m < n$. See \Cref{fig:ex}.
\end{itemize}

\begin{figure}[H]
\centering
\begin{tikzpicture}
\foreach \x in {0,1,...,5}
{
\draw[dashed] (2*\x,0)--(2*\x,2);
\draw[->, dashed](2*\x,0)--(2*\x,.85);
\draw[fill] (2*\x,0) circle [radius=0.05];
\draw[fill] (2*\x,2) circle [radius=0.05];
\node [below] at (2*\x, -0.15) {$\x$};
\node [above] at (2*\x, 2.1) {$\x$};
}
\draw (0,0)--(2,2);
\draw[->](0,0)--(.85,.85);
\draw(.85,.85)--(2,2);
\draw (2,0)--(0,2);
\draw[->](2,0)--(1.15,.85);
\draw(1.15,.85)--(0,2);
\draw (4,0)--(0,2);
\draw[->](4,0)--(2.3,.85);
\draw(2.3,.85)--(0,2);
\draw (6,0)--(2,2);
\draw[->](6,0)--(4.3,.85);
\draw(4.3,.85)--(2,2);
\draw (10,0)--(4,2);
\draw[->](10,0)--(7.45,.85);
\draw(7.45,.85)--(4,2);
\draw (8,0)--(4,2);
\draw[->](8,0)--(6.3,.85);
\draw(6.3,.85)--(4,2);
\draw (10.25,1.75/3)--(6,2);
\draw[->](10.25,1.75/3)--(9.45,.85);
\draw(9.45,.85)--(6,2);
\draw (10.25,3.75/4)--(6,2);
\draw (10.25,11.5/8)--(8,2);
\draw (10.25,15.5/10)--(8,2);
\draw (10.25,19.5/10)--(10,2);
\draw (10.25,23.5/12)--(10,2);
\draw[fill] (10.5,1) circle [radius=0.02];
\draw[fill] (11,1) circle [radius=0.02];
\draw[fill] (11.5,1) circle [radius=0.02];
\end{tikzpicture}
\caption{The algebra $\algebra{F}$. Dashed lines for $f$, solid lines for $g$.}\label{fig:ex}
\end{figure}
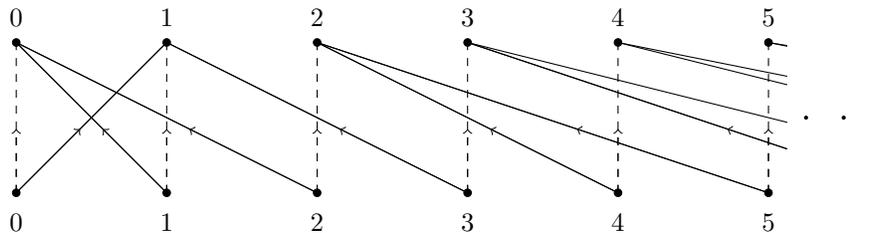

Since $\algebra{F}$ \emph{is} an algebra of partial functions, it is certainly representable by partial functions. It is easy to see that $\algebra{F}$ cannot be represented over a finite base. Indeed, $\R(f) = \U(f)$, so in any representation $f$ is a bijection from its domain, the $d$-vertices, to its range, the $r$-vertices. On the other hand, $\R(g) \neq \U(g)$ so $g$ maps the $d$-vertices onto the $r$-vertices, but not injectively. Hence these sets of vertices cannot have finite cardinality.
\end{example}

By including the operation $\U$ in less expressive signatures, it is possible to give slightly simpler examples than \Cref{example}. However we chose an expansion of the signature $\{\compo, \bmeet, \D, \R\}$ in order to contrast with the other expansions that are the subject of this paper, for which we have seen that the finite representation property does hold.

Note that our example allows us to observe the finite representation property behaving non monotonically as a function of expressivity. Indeed $\U$ is expressible in terms of domain and opposite, $\U(f) = \D(f^{-1})$, and so we have
\[
\{\compo, \bmeet, \D, \R\} \subset \{\compo, \bmeet, \D, \R, \U\} \subset \{\compo, \bmeet, \D, \R, \phantom{}^{-1}, \U\}
\]
with the finite representation property holding for the outer two signatures, but failing in the middle.

\bibliographystyle{amsplain}

\bibliography{finite_representation_property}

\end{document}